\documentclass[10pt,twoside]{amsart}
\usepackage{amsmath}
\usepackage{amssymb}
\usepackage[arrow, matrix, curve]{xy}
\usepackage{hyperref}

\numberwithin{equation}{section}
\numberwithin{table}{section}

\theoremstyle{plain}

\newtheorem{thm}{Theorem}[section]
\newtheorem{thm*}{Theorem}
\newtheorem{lem}[thm]{Lemma}
\newtheorem{prop}[thm]{Proposition}

\theoremstyle{definition}

\newtheorem{defi}[thm]{Definition}
\newtheorem{exa}[thm]{Example}
\newtheorem{setup}[thm]{Setup}

\newtheorem{rem}[thm]{Remark}
\newtheorem{con}[thm]{Construction}

\newcommand{\mm}{\mathfrak m}
\newcommand{\C}{\mathbb{C}}
\newcommand{\D}{\mathbb{D}}
\newcommand{\I}{\mathbb{I}}
\newcommand{\N}{\mathbb{N}}
\newcommand{\Obb}{\mathbb{Obb}}
\newcommand{\T}{\mathbb{T}}
\newcommand{\Z}{\mathbb{Z}}
\newcommand{\Ac}{\mathcal{A}}
\newcommand{\Dc}{\mathcal{D}}
\newcommand{\Ec}{\mathcal{E}}
\newcommand{\Fc}{\mathcal{F}}
\newcommand{\Ic}{\mathcal{I}}
\newcommand{\Mcc}{\mathcal{M}}
\newcommand{\Nc}{\mathcal{N}}
\newcommand{\Lc}{\mathcal{L}}
\newcommand{\Oc}{\mathcal{O}}
\newcommand{\Rc}{\mathcal{R}}
\newcommand{\Sc}{\mathcal{S}}
\newcommand{\Tc}{\mathcal{T}}

\newcommand{\Th}{\text{h}}
\newcommand{\TF}{\text{F}}

\DeclareMathOperator{\chara}{char}
\DeclareMathOperator{\dirsum}{\oplus}
\DeclareMathOperator{\Ima}{im}
\DeclareMathOperator{\Ker}{ker}
\DeclareMathOperator{\coKer}{coker}
\DeclareMathOperator{\Dim}{dim}
\DeclareMathOperator{\Det}{det}
\DeclareMathOperator{\Spec}{Spec}
\DeclareMathOperator{\lK}{H}
\DeclareMathOperator{\Syz}{Syz}
\DeclareMathOperator{\Proj}{Proj}
\DeclareMathOperator{\Deg}{deg}
\DeclareMathOperator{\SL}{SL}
\DeclareMathOperator{\HKF}{HK}
\DeclareMathOperator{\HKM}{e_{HK}}
\DeclareMathOperator{\FS}{F_{sign}}
\DeclareMathOperator{\Mat}{Mat}

\newcommand{\qpot}{^{[q]}}
\newcommand{\pb}{^{\ast}}
\newcommand{\fpb}[1]{\text{F}^{#1\ast}}
\newcommand{\inj}{\hookrightarrow}
\newcommand{\ra}{\rightarrow}
\newcommand{\lra}{\longrightarrow}
\newcommand{\trans}{^{\text{T}}}
\newcommand{\dimglo}[1]{\text{h}^0(#1)}
\newcommand{\punctured}[2]{\Spec(#1)\setminus\{#2\}}

\begin{document}

\subjclass[2010]{
13A35, 
13A50, 
13C14, 
13D02, 
13D40, 
14J17, 
14J60, 
20G05  
}

\keywords{ADE singularity, Hilbert-Kunz function, vector bundle, Frobenius, Hilbert-series, syzygy module, matrix factorization, maximal Cohen-Macaulay}

\author{Daniel Brinkmann}

\address{Institut f\"ur Mathematik \\ Universit\"at Osnabr\"uck}

\email{dabrinkm@uos.de}

\title{The Hilbert-Kunz functions of two-dimensional rings of type ADE}

\date{\today}

\begin{abstract}
We compute the Hilbert-Kunz functions of two-dimensional rings of type ADE by using representations of their indecomposable, maximal 
Cohen-Macaulay modules in terms of matrix factorizations, and as first syzygy modules of homogeneous ideals.
\end{abstract}

\maketitle

\section*{Introduction}
The central objects in this paper will be the two-dimensional \textit{rings of type ADE} (or $ADE$-rings for short), namely
\begin{itemize}
\item $A_n:=k[X,Y,Z]/(X^{n+1}-YZ)$ with $n\geq 0$,
\item $D_n:=k[X,Y,Z]/(X^2+Y^{n-1}+YZ^2)$ with $n\geq 4$,
\item $E_6:=k[X,Y,Z]/(X^2+Y^3+Z^4)$,
\item $E_7:=k[X,Y,Z]/(X^2+Y^3+YZ^3)$ and
\item $E_8:=k[X,Y,Z]/(X^2+Y^3+Z^5)$,
\end{itemize}
as well as their $(X,Y,Z)$-adic completions, where $k$ denotes an algebraically closed field. These rings were studied by several authors, e.g. Klein (in \cite{klein}), du Val (in \cite{duval1}, \cite{duval2}, \cite{duval3}), Brieskorn (in 
\cite{brieskorn}) or Artin (in \cite{artin1}, \cite{artin2}), where they appeared in various forms, e.g as quotient singularities or rational 
double points. They also appear in string theory (cf. \cite{string} for a survey).

The goal is to compute (in positive characteristic) their \textit{Hilbert-Kunz functions}
$$e\mapsto \Dim_k\left(k[X,Y,Z]/\left(F,X^{p^e},Y^{p^e},Z^{p^e}\right)\right),$$
where $F$ denotes one of the defining polynomials above. Note that the (non-local) rings of type $ADE$ are homogeneous with respect to some positive grading.

Recall that the rings $\C[X,Y,Z]/(F)$ of type $ADE$ appear as rings of invariants of $\C[x,y]$ by the actions of the finite subgroups of $\SL_2(\C)$. 
The groups corresponding to the singularities of type $A_n$, $D_n$, $E_6$, $E_7$ resp. $E_8$ are the cylic group with $n+1$ elements, the binary 
dihedral group $\D_{n-2}$ of order $4n-8$, the binary tetrahedral group $\T$ of order 24, the binary octahedral group $\Obb$ of order 48 resp. the binary 
icosahedral group $\I$ of order 120. If $k$ is algebraically closed of characteristic $p>0$ the groups above can be viewed as finite subgroups of $\SL_2(k)$, provided their order 
is invertible in $k$. In these cases $k[X,Y,Z]/(F)$ is again the ring of invariants of $k[x,y]$ under the action of the corresponding group 
(cf. \cite[Chapter 6, \S 2]{leuwie}). Using this fact, Watanabe and Yoshida showed in \cite{watyo1} that the Hilbert-Kunz multiplicity of rings of type 
$ADE$ is given by $2-\tfrac{1}{|G|}$, provided the order of the corresponding group $G$ is invertible modulo $p$. Moreover, by a result of Brenner 
(cf. \cite{bredim2func}) the Hilbert-Kunz functions are of the form
$$e\mapsto \HKM\cdot p^{2e}+\gamma(p^e),$$
where $\gamma$ is an eventually periodic function. For example, the Hilbert-Kunz function of a ring of type $A_n$ is due to Kunz 
(cf. \cite[Example 4.3]{kunz2}) given by
\begin{equation*}
e\mapsto \left(2-\frac{1}{n+1}\right)p^{2e}-r+\frac{r^2}{n+1},
\end{equation*}
where $r$ is the smallest non-negative representative of the class of $p^e$ modulo $n+1$. We also want to mention that the Hilbert-Kunz functions of rings 
of type $E_6$ and $E_8$ can be computed with the algorithm of Han and Monsky (cf. \cite{monskyhan}), provided the characteristic is known.

If $R$ denotes an $ADE$-ring, the flat map 
$$\bar{\phi}:R\ra k[U,V,W]/(\phi(F)),X\mapsto U^{\Deg(X)},Y\mapsto V^{\Deg(Y)},Z\mapsto W^{\Deg(Z)},$$
will allow us to switch between the non-standard graded situation we are interested in and a standard-graded situation. Explictly, the corresponding 
Hilbert-Kunz functions in these two cases are related in the following way (cf. \cite[Remark 1.22]{drdaniel})
\begin{align*}
 & \dim_k\left(R/\left(X^{p^e},Y^{p^e},Z^{p^e}\right)\right)\\
=& \frac{\dim_k\left(k[U,V,W]/\left(\phi(F),U^{\Deg(X)\cdot p^e},V^{\Deg(Y)\cdot p^e},W^{\Deg(Z)\cdot p^e}\right)\right)}{\Deg(X)\cdot\Deg(Y)\cdot\Deg(Z)}.
\end{align*}
This makes it possible to use strong tools on both sides. Namely, in the non-standard graded situation, the $ADE$-rings are 
\textit{graded Cohen-Macaulay finite} (cf. \cite{knoerrer} and \cite{buchmcm}), meaning that there are up to isomorphism 
and to degree shifts only finitely many indecomposable, graded, maximal Cohen-Macaulay modules. In the standard-graded case we can use the geometric interpretation 
of Hilbert-Kunz functions due to Brenner and Trivedi (cf. \cite{bredim2} resp. \cite{tri1}) and the theory of vector bundles.

To be more precise, instead of computing the Hilbert-Kunz functions of the rings $R$ of type $ADE$ (with respect to the ideal $(X,Y,Z)$), we might 
compute the Hilbert-Kunz function of $S:=k[U,V,W]/(\phi(F))$ with respect to the monomial ideal $I:=(U^{\Deg(X)},V^{\Deg(Y)},W^{\Deg(Z)})$. For this purpose, 
we need to control the Frobenius pull-backs of $\Syz_{\Proj(S)}(I)$. The idea for this is the following. Since $R$ is graded Cohen-Macaulay finite 
and $\Syz_R(X,Y,Z)$ as well as all $\Syz_R(X^{p^e},Y^{p^e},Z^{p^e})$, $e\geq 1$, are maximal Cohen-Macaulay there has to be an eventually periodicity 
in the sense that the isomorphism class of $\Syz_R(X^{p^e},Y^{p^e},Z^{p^e})$, $e\gg 0$, depends only on the residue class of $p^e$ modulo a certain number. This 
periodicity of the $R$-modules $\Syz_R(X^{p^e},Y^{p^e},Z^{p^e})$ translates via $\bar{\phi}$ to a periodicity of the $S$-modules 
$\Syz_S(U^{\Deg(X)p^e},V^{\Deg(Y)p^e},W^{\Deg(Z)p^e})$, which makes it possible to compute the Hilbert-Kunz function of $S$ with respect to $I$.

In order to detect this periodicity, we have to find a full list of representatives of the isomorphism classes of the indecomposable, maximal Cohen-Macaulay 
$R$-modules as well as an (easy computable) invariant that allows us to distinguish between these isomorphism classes.

The full list of representatives of the isomorphism classes of the indecomposable, maximal Cohen-Macaulay $R$-modules should consist of first syzygy 
modules of ideals. This is done with the help of Eisenbud's theory of \textit{matrix factorizations}, which encodes maximal Cohen-Macaulay modules over hypersurface rings 
as a pair of matrices. We will see how one might translate a matrix factorization into a first syzygy module of an ideal. This will give the desired 
list of first syzygy modules of ideals. For example in the case of a ring of type $E_8$, we obtain the following theorem.

\begin{thm*}
Let $R:=k[X,Y,Z]/(X^2+Y^3+Z^5).$ The pairwise non-isomorphic modules 
\begin{align*}
M_1 &= \Syz_R(X,Y,Z),\\
M_2 &= \Syz_R(X,Y^2,YZ,Z^2),\\
M_3 &= \Syz_R(XY,XZ,Y^2,YZ^2,Z^3),\\
M_4 &= \Syz_R(XY,XZ^2,Y^3,Y^2Z,YZ^3,Z^4),\\
M_5 &= \Syz_R(XY^2,XYZ^2,XZ^4,Y^4,Y^3Z,Y^2Z^3,Z^5),\\
M_6 &= \Syz_R(X,Y^2,YZ,Z^3),\\
M_7 &= \Syz_R(XY,XZ,Y^2,YZ^2,Z^4),\\
M_8 &= \Syz_R(X,Y,Z^2)
\end{align*}
give a complete list of representatives of the isomorphism classes of indecomposable, non-free, maximal Cohen-Macaulay modules. 
\end{thm*}

The second step is to find an invariant of modules which is easy to calculate and detects the isomorphism 
class of the module. Computing the Hilbert-series of the syzygy modules from our finite list, one recognizes that the Hilbert-series might be a good 
candidate for this invariant. But we will have two further problems, namely we need an algorithm to compute the Hilbert-series of 
$\Syz_R(X^{p^e},Y^{p^e},Z^{p^e})$, which depends on the parameters $p$ and $e$ and the second problem is that in the cases where $R$ is of type $D$ or $E$ 
the Hilbert-series does not detect whether a given module is indecomposable or not.

The problem that we need to know that the modules $\Syz_R(X^{p^e},Y^{p^e},Z^{p^e})$ are indecomposable is solved by using ringinclusions of the form 
$\iota: R\ra A_n$, where $R$ is of type $D$ or $E$. Translating everything to the corresponding punctured spectra $\Rc$ and $\Ac_n$, we can 
pull the $\Oc_{\Rc}$-modules in question back to $\Ac_n$ using the map induced by $\iota$. Since every $\Oc_{\Ac_n}$-module is a direct sum of line 
bundles and since we know how the Frobenius acts on line bundles, we will deduce from the commutativity of the pull-backs via the Frobenius and the map induced 
by $\iota$, that all $\Syz_R(X^{p^e},Y^{p^e},Z^{p^e})$ are indecomposable.

Last but not least we need to deal with the problem of computing the Hilbert-series of $\Syz_R(X^{p^e},Y^{p^e},Z^{p^e})$. To do so, we will prove 
the following theorem.

\begin{thm*}
Let $k$ be a field and $$R:=k[X,Y,Z]/(X^d-F(Y,Z))$$ with $\Deg(X)=\alpha$, $\Deg(Y)=\beta$, $\Deg(Z)=\gamma$ and $F\in k[Y,Z]$ homogeneous of degree 
$d\alpha$ ($d$, $\alpha$, $\beta$, $\gamma\in\N_{\geq 1}$). Let $a\in\N_{\geq 1}$ and write $a=d\cdot q +r$ with $0\leq r\leq d-1$ and $q\in\N$.
For $l\in\N$ we use the abbreviation $\Sc(l):=\Syz_R(X^l,Y^b,Z^c)$. Then the Hilbert-series of $\Sc(a)=\Syz_R(X^a,Y^b,Z^c)$ is given by
\begin{equation*}
\lK_{\Sc(a)}(t) = \frac{(t^{\alpha\cdot r}-t^{\alpha\cdot d})\cdot \lK_{\Sc(a-r)}(t)+(1-t^{\alpha\cdot r})\lK_{\Sc(a+d-r)}(t)}{1-t^{\alpha\cdot d}}.
\end{equation*}
\end{thm*}

By iteration this theorem reduces the task of computing the Hilbert-series of the modules $\Syz_R(X^{p^e},Y^{p^e},Z^{p^e})$ to the computation of the Hilbert-series of the 
modules $\Syz_R(X^{p^e-1},Y^{p^e},Z^{p^e})$ and $\Syz_R(X^{p^e+1},Y^{p^e},Z^{p^e})$ for only finitely many small values of $e$, which can be done over the polynomial ring $k[Y,Z]$. Over 
$P:=k[Y,Z]$ these modules split as a direct sum $P(m)\oplus P(n)$ for some $m$, $n\in\Z$ by Hilbert's Syzygy Theorem. To compute the number $|n-m|$ 
we will define a non-standard graded analogue $\tau$ of Han's $\delta$-function and prove a connection between $\tau$ and $\delta$ which allows us 
to compute $\tau$ effectively. To give an example, assume $R$ is of type $E_8$. Then $\Syz_R(X^{2a},Y^{b},Z^{c})$ with non-negative $a,b,c\in\Z$ splits as 
$k[Y,Z](m)\oplus k[Y,Z](n)$ and we have 
$$|m-n|=30\cdot\delta\left(a,\frac{b}{3},\frac{c}{5}\right).$$

At the end of this introduction, we want to give an outline of this article.

In Section 1 we give some basic notations and discuss what is known about the Hilbert-Kunz functions of two-dimensional rings of type $ADE$. 

In Section 2 the process how one gets syzygy modules out of matrix factorizations will be explained.

In the third section we will prove that the modules $\Syz_R(X^{p^e},Y^{p^e},Z^{p^e})$ are indecomposable in almost all characteristics if $R$ is of type 
$D$ or $E$.

In Section 4 we explain how the Hilbert-series of the modules $\Syz_R(X^{p^e},Y^{p^e},Z^{p^e})$ can be computed.

Finally, in Section 5 we will see how everything works together to obtain the Hilbert-Kunz functions of rings of type $ADE$.

We will use the ring of type $E_8$ as a running example to illustrate the various steps in the computation of the Hilbert-Kunz functions of 
two-dimensional $ADE$-rings.

\section*{Acknowledgements}
Since the results in this paper are part of my PhD-thesis, I thank first of all my advisor Holger Brenner for his patience and friendship during 
the last years.

I thank Almar Kaid and Axel St\"abler for many usefull discussions on Hilbert-Kunz theory as well as on vector bundles and 
Alessio Caminata for the discussions on the proof that two-dimensional $ADE$-rings are Cohen-Macaulay finite.

Moreover, I want to thank Igor Burban who brought the theory of matrix factorizations to my attention and Kevin Tucker
who is responsible for the corollary about the $F$-signature functions of the rings of type $ADE$.

\section{Hilbert-Kunz functions via vector bundles}
\begin{defi}
Let $(R,\mm)$ be a local Noetherian ring of characteristic $p>0$ and dimension $d$. Let $I=(f_1,\ldots,f_m)$ be an $\mm$-primary ideal and 
let $\lambda_R$ denote the length function.
\begin{enumerate} 
\item The function $\HKF(I,R,p^e):\N\ra\N$, 
$$e\longmapsto\lambda_R\left(R/\left(f_1^{p^e},\ldots,f_m^{p^e}\right)\right)$$
is called the \textit{Hilbert-Kunz function of $R$ with respect to $I$}.
\item We call the limit
$$\lim_{e\ra\infty}\frac{\HKF(I,R,p^e)}{p^{ed}},$$
whose existence was proven by Monsky in \cite{hkmexists}, the \textit{Hilbert-Kunz multiplicity of $R$ with respect to $I$}. We denote this 
limit by $\HKM(I,R)$ and set $\HKM(R):=\HKM(\mm,R)$ for short. We will call $\HKM(R)$ the Hilbert-Kunz multiplicity of $R$.
\end{enumerate}
\end{defi}

Sometimes we will use the symbol $(f_1,\ldots,f_m)^{[p^e]}$ to denote the ideal $(f_1^{p^e},\ldots,f_m^{p^e})$.

\begin{rem}
Note that for any finitely generated $R$-module $M$ its length can be computed as $\lambda_R(M)=\tfrac{\Dim_k(M)}{\Dim_k(R/\mm)}$ if $R$ is an algebra over some field $k$ and $\Dim_k(R/\mm)$ is finite.
\end{rem}

Next we give some basic facts on syzygy bundles, which are the vector bundles that will appear in the application of vector bundles to Hilbert-Kunz theory.

\begin{defi}
Let $k$ be a field and $R$ an affine $k$-algebra. Let $f_1,\ldots,f_n\in R$ and $X:=\Spec(R)$. 
\begin{enumerate}
\item We call the kernel of the map 
 $$\Oc_X^n\stackrel{f_1,\ldots,f_n}{\lra}\Oc_X$$
 the \textit{sheaf of syzygies} for $f_1,\ldots,f_n$ and denote it by $\Syz_X(f_1,\ldots,f_n)$. 
\item Let $\Ic_Z$ be the ideal sheaf of the closed subscheme $Z:=V(f_1,\ldots,f_n)\subseteq X$. We call the short exact sequence
 $$0\lra\Syz_X(f_1,\ldots,f_n)\lra\Oc_X^n\stackrel{f_1,\ldots,f_n}{\lra}\Ic_Z\lra 0$$
 the \textit{presenting sequence} of the sheaf of syzygies.
\end{enumerate}
\end{defi}

\begin{defi}
Let $k$ be a field and $R$ a standard-graded, affine $k$-algebra. Let $f_1,\ldots,f_n\in R$ be homogeneous elements with $\Deg(f_i)=d_i$. Let $Y:=\Proj(R)$. 
\begin{enumerate}
\item We call the kernel of the map 
 $$\bigoplus_{i=1}^n\Oc_Y(-d_i)\stackrel{f_1,\ldots,f_n}{\lra}\Oc_Y$$
 the \textit{sheaf of syzygies} for $f_1,\ldots,f_n$ and denote it by $\Syz_Y(f_1,\ldots,f_n)$. 
\item Let $\Ic_Z$ be the ideal sheaf of the closed subscheme $Z:=V_+(f_1,\ldots,f_n)\subseteq Y$. We call the short exact sequence
 $$0\lra\Syz_Y(f_1,\ldots,f_n)\lra\bigoplus_{i=1}^n\Oc_Y(-d_i)\stackrel{f_1,\ldots,f_n}{\lra}\Ic_Z\lra 0$$
 the \textit{presenting sequence} of the sheaf of syzygies.
\end{enumerate}
\end{defi}

Note that the sheaf $\Syz_X(f_1,\ldots,f_n)$ in the affine case as well as $\Syz_Y(f_1,\ldots,f_n)$ in the projective case are nothing but the sheafification of the (graded) 
$R$-module $\Syz_R(f_1,\ldots,f_n)$. Therefore we may define in the graded case for any $m\in\Z$ the \textit{sheaf of syzygies of total degree $m$} of 
$f_1,\ldots,f_n$, denoted by $\Syz_Y(f_1,\ldots,f_n)(m)$, as the sheafification of the $R$-module $\Syz_R(f_1,\ldots,f_n)(m)$. Note that the equality
$$\Syz_Y(f_1,\ldots,f_n)(m)=\Syz_Y(f_1,\ldots,f_n)\otimes\Oc_Y(m)$$ 
holds, since $R$ is standard-graded. Moreover, if 
$R$ has dimension at least two and satisfies $S_2$ the sheaf 
$\Syz_Y(f_1,\ldots,f_n)(m)$ encodes the homogeneous component of degree $m$ of the $R$-module $\Syz_R(f_1,\ldots,f_n)$. 
We get this homogeneous component back by taking global sections (cf. \cite[Exercise III.3.5]{hartshorne})
$$\Gamma\left(Y,\Syz_Y(f_1,\ldots,f_n)(m)\right)\cong \Syz_R(f_1,\ldots,f_n)_m.$$

Recall some basic properties of syzygy modules. The proofs might be found in \cite[Chapter 1.2.3]{drdaniel}.

\begin{prop}\label{manipulationlemma} Let $R$ be a standard-graded, affine $k$-domain and pick some homogeneous elements 
$f,f_1,\ldots,f_n\in R$ with $f\neq 0$.
\begin{enumerate}
 \item The following isomorphism holds
 $$\Syz_R(ff_1,\ldots,ff_n)\cong\Syz_R(f_1,\ldots,f_n)(-\Deg(f)).$$
 \item If $f,f_n$ is an $R$-regular sequence, we have 
 $$\Syz_R(ff_1,\ldots,ff_{n-1},f_n)\cong\Syz_R(f_1,\ldots,f_n)(-\Deg(f)).$$
 \item If $\Deg(ff_i)=\Deg(f_n)$ for some $i\in\{1,\ldots,n-1\}$, we have $$\Syz_R(f_1,\ldots,f_n)\cong\Syz_R(f_1,\ldots,f_{n-1},f_n+ff_i).$$ 
 \end{enumerate}
\end{prop}

\begin{prop}\label{syzprop}
Let $k$ be an algebraically closed field and $R$ a standard-graded $k$-algebra of dimension at least two and $Y:=\Proj(R)$. Let 
$f_1,\ldots,f_n\in R$ be homogeneous elements with $\Deg(f_i)=d_i$. Assume that at least one of the $f_i$ is a non-zero 
divisor. Then the following statements hold.

\begin{enumerate}
\item The rank of $\Syz_Y(f_1,\ldots,f_n)$ is $n-1$.
\item If the ideal generated by the $f_i$ is $R_+$-primary, then the morphism 
 $$\bigoplus_{i=1}^n\Oc_Y(-d_i)\stackrel{f_1,\ldots,f_n}{\lra}\Oc_Y$$
 is surjective and $\Syz_Y(f_1,\ldots,f_n)$ is locally free.
\item The sheaf $\Syz_Y(f_1,\ldots,f_n)$ is locally free on $U:=D_+(f_1,\ldots,f_n)\subseteq Y$.
\item If $Y$ is an irreducible curve and the ideal $(f_1,\ldots,f_n)$ is $R_+$-primary, then
 $$\Det(\Syz_Y(f_1,\ldots,f_n)(m))\cong\Oc_Y\left((n-1)m-\sum_{i=1}^nd_i\right).$$
 In particular, we have $$\Deg(\Syz_Y(f_1,\ldots,f_n)(m))=\left((n-1)m-\sum_{i=1}^nd_i\right)\cdot \Deg(Y).$$
\item Let $\phi:R\ra S$ be a morphism of degree $d$ of normal, standard-graded $k$-domains of dimension two. Let $g:Y:=\Proj(S)\ra X$ be the 
 corresponding morphism of smooth projective curves. Then the following holds 
 $$g\pb(\Syz_X(f_1,\ldots,f_n)(m))\cong\Syz_Y(\phi(f_1),\ldots,\phi(f_n))(dm).$$
\item Assume that $\chara(k)=p>0$, that $Y$ is an integral curve and that the ideal $(f_1,\ldots,f_n)$ is $R_+$-primary. For the iterated pull-backs under the Frobenius, we have
 $$\fpb{e}(\Syz_X(f_1,\ldots,f_n)(m))\cong\Syz_X(f_1^q,\ldots,f_n^q)(mq),$$
 where $q=p^e$ and $e\in\N$.
\end{enumerate}
\end{prop}

The next lemma shows how syzygy bundles arise in Hilbert-Kunz theory. For a proof see \cite{bredim2} or \cite{tri1}.

\begin{lem}\label{hkgeomapp}
Let $k$ be an algebraically closed field of positive characteristic $p$, let $R$ be a standard-graded $k$-domain of dimension at least two satisfying $S_2$ and 
let $Y=\Proj(R)$. Let $I=(f_1,\ldots,f_n)$ be a homogeneous $R_+$-primary ideal with $\Deg(f_i)=d_i$. Then for every $q=p^e\geq 1$ and 
$m\in\Z$ the following formula holds
\begin{align*}
\begin{split}
\Dim_k\left(R/\left(f_1^{p^e},\ldots,f_n^{p^e}\right)\right)_m =& \Th^0(X,\Oc_X(m))-\sum_{i=1}^n\Th^0(X,\Oc_X(m-p^e\cdot d_i))\\
 & +\Th^0\left(X,\Syz_X\left(f_1^{p^e},\ldots,f_n^{p^e}\right)(m)\right).
\end{split}
\end{align*}
\end{lem}

If $R=k[X,Y,Z]/(F)$ is a standard-graded, two-dimensional, normal domain and $I$ is generated by powers of $X$, $Y$ and $Z$, we will use Lemma 
\ref{hkgeomapp} to provide explicit formulas for $\HKF(I,R,p^e)$ in the two cases where $\Syz_{\Proj(R)}(I)$ admits a Frobenius periodicity up to twist 
resp. where the higher Frobenius pull-backs of $\Syz_{\Proj(R)}(I)$ split as a direct sum of twisted structure sheafs.

\begin{defi}
Let $X$ be a scheme over an algebraically closed field $k$ of characteristic $p>0$. Let $\Sc$ be a vector bundle over $X$. 
Assume there are $s<t\in\N$ such that the Frobenius pull-backs $\fpb{e}(\Sc)$ of $\Sc$ are pairwise non-isomorphic 
for $0\leq e\leq t-1$ and $\fpb{t}(\Sc)\cong\fpb{s}(\Sc)$. We say that $\Sc$ admits a \textit{$(s,t)$-Frobenius periodicity.} 
Finally, the bundle $\Sc$ admits a \textit{Frobenius periodicity} if there are $s<t\in\N$ such that $\Sc$ admits a $(s,t)$-Frobenius periodicity.
\end{defi}

\begin{rem}
If $X$ is a projective scheme only bundles of degree zero might admit a Frobenius periodicity in the strong sence above. For bundles of non-zero degree 
we will use the following weaker definition of Frobenius periodicity. We say that $\Sc$ admits a $(s,t)$-Frobenius periodicity if 
$$\fpb{e}(\Sc)\ncong\fpb{e'}(\Sc)(m) \quad\text{and}\quad \fpb{t}(\Sc)\cong\fpb{s}(\Sc)(n)$$
hold for all $0\leq e'<e\leq t-1$, all $m\in\Z$ and some $n\in\Z$.
\end{rem}

\begin{lem}\label{syznontr} 
Let $R:=k[X,Y,Z]/(F)$ be a normal standard-graded domain over an algebraically closed field $k$ with $\chara(k)=p>0$. Fix the positive natural numbers 
$\alpha$, $\beta$, $\gamma$ as well as the ideal $I=(X^{\alpha},Y^{\beta},Z^{\gamma})$. Assume that for some $e\in\N$ one has
\begin{align}\label{eqsyzsyz}\Syz_R\left(I^{[p^e]}\right)\cong\Syz_R(X^a,Y^b,Z^c)(-n)\end{align}
for some $n\in\Z$ and some positive integers $a$, $b$, $c$ such that at least one of the inequalities $a<\alpha p^e$, $b<\beta p^e$, $c<\gamma p^e$ holds. 
With $D:=\Dim_k(R/(X^a,Y^b,Z^c))$ the value at $e$ of the Hilbert-Kunz function of $R$ with respect to $I$ is given by
$$\HKF(I,R,p^e)=\frac{d\cdot Q(\alpha,\beta,\gamma)\cdot p^{2e}-d\cdot Q(a,b,c)}{4}+D,$$
where $Q$ denotes the quadratic form 
$$(a,b,c)\longmapsto 2(ab+ac+bc)-a^2-b^2-c^2.$$
\end{lem}

\begin{proof}
Computing the degrees of the $C:=\Proj(R)$-bundles corresponding to the $R$-modules in (\ref{eqsyzsyz}), we find
$$n=\frac{q(\alpha+\beta+\gamma)-a-b-c}{2}.$$
By assumption we have the following diagram

$$\xymatrix{
0\ar[r] & \Syz_R(I\qpot)\ar[r]\ar[d]^{\cong} & R(-\alpha\cdot q)\oplus R(-\beta\cdot q)\oplus R(-\gamma\cdot q)\ar[r] & I\qpot\ar[r] & 0 \\
0\ar[r] & \Syz_R(J)\ar[r] & R(-n-a)\oplus R(-n-b)\oplus R(-n-c)\ar[r] & J\ar[r] & 0,
}$$
with $J:=(X^a,Y^b,Z^c)$. Sheafifying the second row, twisting by $\Oc_C(m)$ and taking global dimensions yields
\begin{align}
\nonumber &\quad \Th^0(C,\Syz_C(X^a,Y^b,Z^c)(m-n))\\ 
&= \Th^0(C,\Oc_C(m-n-a))+\Th^0(C,\Oc_C(m-n-b))+\Th^0(C,\Oc_C(m-n-c)) \label{subst} \\
\nonumber &\quad -\Th^0(C,\Oc_C(m-n))+\Dim_k(R/(X^a,Y^b,Z^c))_{m-n}.
\end{align}
Doing the same with the first row and substituting the term 
$$\Th^0\left(C,\Syz_C\left(I\qpot\right)(m)\right)$$
with the right hand side of (\ref{subst}), we get
\begin{align}\label{eqglosec}
\nonumber &\quad \Dim_k(R/(X^{\alpha q},Y^{\beta q},Z^{\gamma q}))\\
\nonumber &= \lim_{x\ra\infty}\left[\sum_{m=0}^x\dimglo{\Oc_C(m)}-\sum_{m=0}^x\dimglo{\Oc_C(m-\alpha q)}-\sum_{m=0}^x\dimglo{\Oc_C(m-\beta q)}\right.\\
\nonumber &\quad -\sum_{m=0}^x\dimglo{\Oc_C(m-\gamma q)}-\sum_{m=0}^x\dimglo{\Oc_C(m-n)}+\sum_{m=0}^x\dimglo{\Oc_C(m-n-a)}\\
\nonumber &\quad \left. +\sum_{m=0}^x\dimglo{\Oc_C(m-n-b)}+\sum_{m=0}^x\dimglo{\Oc_C(m-n-c)}\right]+D\\
\nonumber &= \lim_{x\ra\infty}\left[\sum_{m=0}^x\dimglo{\Oc_C(m)}-\sum_{m=0}^{x-\alpha q}\dimglo{\Oc_C(m)}-\sum_{m=0}^{x-\beta q}\dimglo{\Oc_C(m}\right.\\
\nonumber &\quad \left. -\sum_{m=0}^{x-\gamma q}\dimglo{\Oc_C(m)}+\sum_{m=0}^{x-n-a}\dimglo{\Oc_C(m)}+\sum_{m=0}^{x-n-b}\dimglo{\Oc_C(m)}\right.\\
\nonumber &\quad \left. +\sum_{m=0}^{x-n-c}\dimglo{\Oc_C(m)}-\sum_{m=0}^{x-n}\dimglo{\Oc_C(m)}\right]+D\\
 &= \lim_{x\ra\infty}\left[\sum_{m=x-n+1}^x\dimglo{\Oc_C(m)}-\sum_{m=x-n-a+1}^{x-\alpha q}\dimglo{\Oc_C(m)}\right.\\
\nonumber &\quad \left. -\sum_{m=x-n-b+1}^{x-\beta q}\dimglo{\Oc_C(m}-\sum_{m=x-n-c+1}^{x-\gamma q}\dimglo{\Oc_C(m)}\right]+D.
\end{align}

By Riemann-Roch we have 
$$\Th^0(\Oc_C(m))=m\cdot \Deg(C)+1-g+\Th^1(\Oc_C(m)),$$
where $g$ is the genus of $C$. Substituting this in Equation (\ref{eqglosec}), the summands involving the genus cancel and the sums over $\Th^1(\Oc_C(m))$ vanish for $x\gg 0$ due to Serre. 
The surviving summands can be expressed as
$$\frac{d\cdot Q(\alpha,\beta,\gamma)\cdot q^2-d\cdot Q(a,b,c)}{4}+D.$$
\end{proof}

Similarly, one can show the following.

\begin{lem}\label{syztr}
With the notations from Lemma \ref{syznontr}, assume that there is an $e\in\N$ such that
\begin{align}\label{syzsplit1}\Syz_R\left(I^{[p^e]}\right)=R(-n)\oplus R(-l)\text{ for some }n,l\in\Z.\end{align}
Then the value at $e$ of the Hilbert-Kunz function of $R$ with respect to $I$ is given by
$$\HKF(I,R,p^e)=\frac{d\cdot Q(\alpha,\beta,\gamma)\cdot p^{2e}+d\cdot|n-l|^2}{4}.$$
Note that the assumption (\ref{syzsplit1}) means that the quotient $R/I^{[p^e]}$ has finite projective dimension and this carries over to all quotients 
$R/I^{[p^{e'}]}$ for $e'\geq e$.
\end{lem}

\section{Matrix factorizations and representations as first syzygy modules of ideals}\label{secmatfac}
Let $S:=k[X,Y,Z]$ and let $R:=S/(F)$ be a ring of type $ADE$. The goal is to construct isomorphisms between maximal Cohen-Macaulay $R$-modules 
and first syzygy modules of homogeneous ideals of $R$. To be more precise, given a full set $M_1,\ldots,M_n$ of representatives for the isomorphism 
classes of indecomposable, maximal Cohen-Macaulay $R$-modules, we want to find for each $i$ non-zero elements $f_1,\ldots,f_m\in R$ such that
$$M_i\cong\Syz_R(f_1,\ldots,f_m).$$

This has to be possible, since first syzygy modules of ideals are second syzygy modules of quotient rings.

We will denote the punctured spectrum of $R$ by $U$. Note that all maximal Cohen-Macaulay $R$-modules are locally free on $U$ since $\Spec(R)$ has 
an isolated singularity at the origin.

Now we describe a way to find a representation of a maximal Cohen-Macaulay module, given by a matrix factorization, as a syzygy bundle on $U$.

At first we recall the definition of a matrix factorization and some general facts. Our reference for matrix factorizations is Chapter 7 of the 
textbook \cite{yobook}. For the interaction of the graded and the local situation see \cite[Chapter 15]{yobook}.

\begin{defi}
Let $S$ be a polynomial ring. Let $\phi,\psi\in \Mat_n(S)$ with $\phi\psi=\psi\phi=F\cdot E_n$, for a non-zero element $F\in (S_+)^2$. Then the pair 
$(\phi,\psi)$ is a \textit{matrix factorization for $F$ of size $n$}. It is called \textit{reduced} if all entries are non-units.
\end{defi}

\begin{rem}
Let $(\phi,\psi)$ be a matrix factorization for some $F$ of size $n$.

\begin{enumerate}
 \item The module $\coKer(\phi)$ is a maximal Cohen-Macaulay $R:=S/(F)$-module, which is non-trivial if and only if $(\phi,\psi)$ is reduced.
 \item If $F$ is irreducible, the determinant of $\phi$ is $u\cdot F^i$, where $u\in S$ is a unit and $0\leq i\leq n$. Moreover, the rank of the $R$-module 
 $\coKer(\phi)$ equals $i$.
 \item The pair $(\phi\trans,\psi\trans)$ is again a matrix factorization for $F$ of size $n$ and 
 $$\coKer(\phi\trans)\cong\coKer(\phi)^{\vee}.$$
 \item It is possible to define direct sums as well as morphisms of matrix factorizations. Then the set of matrix factorizations for $F$ becomes an 
 additive category and Eisenbud has shown in \cite{eismat} that the equivalence classes of reduced, indecomposable matrix factorizations are in one 
 to one correspondence with the isomorphism classes of non-free, indecomposable, maximal Cohen-Macaulay $R$-modules.
 \end{enumerate}
\end{rem}

\begin{con}
For a (reduced) matrix factorization $(\phi_S,\psi_S)$ of size $n$, we have a two periodic free resolution of $M:=\coKer(\phi_S)$ as an $R$-module:
$$\ldots\stackrel{\psi}{\lra} R^n\stackrel{\phi}{\lra} R^n\stackrel{\psi}{\lra} R^n\stackrel{\phi}{\lra} R^n\lra M\lra 0,$$
where $\phi$ and $\psi$ arise from $\phi_S$ and $\psi_S$ by taking their entries modulo $F$. Note that we can identify $M$ with $\coKer(\phi)$.

Therefore we get $$M\cong R^n/\Ima(\phi)\cong R^n/\Ker(\psi)\cong \Ima(\psi).$$

For any subset $J$ of $\{1,\ldots,n\}$ we denote by $\psi^J$ the matrix obtained from $\psi$ by keeping all columns whose index belongs to $J$ (and 
deleting all other columns).

If $\coKer(\phi)\cong \Ima(\psi)$ has rank $m$ as an $R$-module, the sheaf $\Ima(\psi)^{\sim}|_U$ gives a locally free $\Oc_U$-module of rank $m$. 
Therefore a representing first syzygy module of an ideal has to be a syzygy module of an $(m+1)$-generated ideal.

Suppose that we can choose a $J\subseteq\{1,\ldots,n\}$ of cardinality $m+1$ such that $\Ima(\psi)^{\sim}|_U$ and $\Fc:=\Ima(\psi^J)^{\sim}|_U$
are isomorphic as sheaves of $\Oc_U$-modules\footnote{Note that we have no argument why this should be possible, but it works in all examples.}, where the isomorphism is the map induced by the natural inclusion 
$\Ima\left(\psi^J\right)\subseteq\Ima(\psi)$. 
A necessary condition for this is that the matrix $\psi^J$ has full rank - namely $m$ - in every non-zero point. 
Otherwise there would be at least one point $u\in U$ such that the stalk $(\Ima(\psi^J)^{\sim})_u$ has rank at most $m-1$.

Since $\Ima(\psi)^{\sim}|_U$ comes from a maximal Cohen-Macaulay $R$-module, it is locally free as $\Oc_U$-module, giving that 
$$\Oc_U^{m+1}\stackrel{\psi^J}{\lra}\Fc\lra 0$$
is a surjection of locally free $\Oc_U$-modules of ranks $m+1$ and $m$. Therefore the kernel $\Lc$ has to be locally free of rank one. This gives an 
isomorphism 
$$\Lc\otimes\Det(\Fc)\cong\Oc_U$$
of the determinants. We distinguish whether the determinant of $\Fc$ is trivial or not.

\textit{Case 1.} Assume $\Det(\Fc)\cong\Oc_U$. Then the line bundle $\Lc$ has to be trivial. Hence
\begin{equation}\label{dettrivial}
0\lra\Oc_U\stackrel{\eta}{\lra}\Oc_U^{m+1}\stackrel{\psi^J}{\lra}\Fc\lra 0
\end{equation}
is a short exact sequence and $\eta$ has to be the multiplication by a column vector $(g_1,\ldots,g_{m+1})\trans$ with $g_i\in R$. But then Sequence 
\eqref{dettrivial} is exactly the dual of the presenting sequence of $\Syz_U(g_1,\ldots,g_{m+1})$, hence there is an isomorphism
$$\Syz_U(g_1,\ldots,g_{m+1})\cong\Fc^{\vee}.$$
Moreover, the ideal generated by the $g_i$ has to be $R_+$-primary, since the determinant of the syzygy bundle is trivial.

\textit{Case 2.} Now assume that $\Lc$ is a non-trivial line bundle, hence $\Det(\Fc)\cong\Lc^{\vee}$. Dualizing the sequence 
$0\ra\Lc\ra\Oc_U^{m+1}\ra\Fc\ra 0$ yields
\begin{equation}\label{detnontrivial}0\lra\Fc^{\vee}\stackrel{(\psi^J)\trans}{\lra}\Oc_U^{m+1}\ra\Lc^{\vee}\lra 0.\end{equation}
Since $\Lc^{\vee}$ is a line bundle on $U$ there exists an embedding $\Lc^{\vee}\hookrightarrow\Oc_U$. With 
this embedding, we can extend Sequence \eqref{detnontrivial} to 
\begin{equation}\label{nontrivial}0\lra\Fc^{\vee}\stackrel{(\psi^J)\trans}{\lra}\Oc_U^{m+1}\stackrel{\mu}{\lra}\Oc_U\lra\Tc\lra 0,\end{equation}
where $\Tc$ is a non-zero torsion sheaf and the map $\mu$ is the multiplication by a row vector $(g_1,\ldots,g_{m+1})$. In this case, Sequence 
\eqref{nontrivial} is the presenting sequence of the syzygy bundle $\Syz_U(g_1,\ldots,g_{m+1})$, where the ideal $(g_1,\ldots,g_{m+1})$ is not 
$R_+$-primary, since $\Tc$ is non-zero.

In both cases the column vector $(g_1,\ldots,g_{m+1})\trans$ generates the kernel of $\psi^J$. In order to compute the $g_i$, we only need to find 
an element $(f_1,\ldots,f_{m+1})\trans\in\Ker\left(\psi^J\right)$. Then $(f_1,\ldots,f_{m+1})\trans$ is just a multiple of $(g_1,\ldots,g_{m+1})\trans$ 
and equality (up to multiplication by a unit) holds if we can prove that the $f_i$ are coprime. For example this condition is automatically fulfilled 
if the ideal generated by the $f_i$ is $R_+$-primary. Anyway, the syzygy modules of the two ideals are isomorphic.

All in all, our task is to find a vector $(f_1,\ldots,f_{m+1})\trans$ in the kernel of $\psi^J$. This gives an isomorphism 
$$\widetilde{M^{\vee}}\cong\Fc^{\vee}\cong\Syz_U(f_1,\ldots,f_{m+1}).$$
This isomorphism extends to a global isomorphism of $R$-modules, since all involved modules are reflexive (cf. \cite[Lemma 3.6]{mcmsur})
\end{con}

We will illustrate this construction by an explicit computation.

\begin{exa}
Let $F=X^2+Y^3+Z^5$ be the equation of an $E_8$-singularity. The pair of matrices
$$(\phi,\psi):=\left(\begin{pmatrix}
X & Y & Z & 0\\
-Y^2 & X & 0 & Z\\
-Z^4 & 0 & X & -Y\\
0 & -Z^4 & Y^2 & X
\end{pmatrix},
\begin{pmatrix}
X & -Y & -Z & 0\\
Y^2 & X & 0 & -Z\\
Z^4 & 0 & X & Y\\
0 & Z^4 & -Y^2 & X
\end{pmatrix}\right)$$
is a reduced (indecomposable) matrix factorization for $F$ of size 4 and the corresponding module $\coKer(\phi)$ has rank two. The following computations 
show that over $U$ the first column of $\psi$ lies in the span generated by the other three columns of $\psi$
\begin{align*}
\frac{Y^2}{X}\cdot \begin{pmatrix} -Y \\ X \\ 0 \\ Z^4 \end{pmatrix}+\frac{Z^4}{X}\cdot \begin{pmatrix} -Z \\ 0 \\ X \\ -Y^2 \end{pmatrix} = &  \begin{pmatrix} X \\ Y^2 \\ Z^4 \\ 0 \end{pmatrix},\\
-\frac{X}{Y}\cdot \begin{pmatrix} -Y \\ X \\ 0 \\ Z^4 \end{pmatrix}+\frac{Z^4}{Y}\cdot \begin{pmatrix} 0 \\ -Z \\ Y \\ X \end{pmatrix} = &  \begin{pmatrix} X \\ Y^2 \\ Z^4 \\ 0 \end{pmatrix}.
\end{align*}
Now we have to find three non-zero elements $f_1$, $f_2$ and $f_3$ such that $(f_1,f_2,f_3)\trans$ belongs to the kernel of $\psi^{\{2,3,4\}}$.

The first row $(-Y,-Z,0)$ of $\psi^{\{2,3,4\}}$ gives the relation $-Y\cdot f_1-Z\cdot f_2=0$ from which we conclude (since $Y$ and $Z$ are coprime 
in $R$) that $f_1=Z\cdot g$ and $f_2=-Y\cdot g$ for some non-zero $g\in R$.
From the second row $(X,0,-Z)$ we get the relation $Z\cdot X\cdot g-Z\cdot f_3=0$. This gives $f_3=X\cdot g$. Looking at the relations given by the 
rows three and four of $\psi^{\{2,3,4\}}$, we see that $g$ can be chosen as 1. We get the isomorphism
$$\Ima(\psi)^{\sim}|_U^{\vee}\cong\Ima(\psi^{\{2,3,4\}})^{\sim}|_U^{\vee}\cong\Syz_U(Z,-Y,X),$$
extending to the global isomorphism
$$\Ima(\psi)^{\vee}\cong\Syz_R(Z,-Y,X).$$
It is easy to see that $\Syz_R(Z,-Y,X)$ is self-dual since $\coKer(\phi)^{\vee}\cong\coKer(\phi\trans)$ and 
$(\alpha,\alpha):(\phi,\psi)\lra(\phi\trans,\psi\trans)$ with
$$\alpha:=\begin{pmatrix} 0 & 0 & 0 & -1\\ 0 & 0 & 1 & 0\\ 0 & -1 & 0 & 0\\ 1 & 0 & 0 & 0\end{pmatrix}$$
is an equivalence of matrix factorizations. Hence,
$$\coKer(\phi)\cong\Ima(\psi)\cong\Syz_R(Z,-Y,X).$$
In fact, one can delete any column from $\psi$, giving the isomorphisms
\begin{align*}
 \Ima(\psi) \cong & \Syz_R(Z,-Y,X)\\
 \cong & \Syz_R(Z,X,Y^2)\\
 \cong & \Syz_R(Y,X,-Z^4)\\
 \cong & \Syz_R(X,-Y^2,-Z^4).
\end{align*}
\end{exa}

Note that if $\coKer(\phi)$ has rank one, the isomorphism $\coKer(\phi)\cong\Ima(\psi)$ is already enough to obtain an isomorphism of the form 
$\coKer(\phi)\cong\Syz_R(f_1,f_2)$, since for a suitable choice of $f_1$, $f_2$ the columns of $\psi$ are directly the generators $\Syz_R(f_1,f_2)$.

A complete list of non-equivalent matrix factorizations representing all isomorphism classes of indecomposable, maximal Cohen-Macaulay $R$-modules can be 
found in \cite{matfacade} resp. \cite[Chapter 9, \S 4]{leuwie}. Using these lists, one obtains the following theorems whose proofs are analogue to the previous examples 
and can be found explicitly in \cite[Chapter 3]{drdaniel}. The statements of the form $\Det(\Syz_R(I)^{\sim}|_U)\cong \Syz_R(J)^{\sim}|_U$ follow 
by showing $I^{\vee\vee} = J$. We should mention that the enumeration of the modules in the subsequent theorems corresponds to the enumeration of 
the matrix factorizations as given in the above references. It corresponds also to the enumeration of the vertices in the corresponding Dynkin diagrams.

\begin{thm}\label{syzan}
Let $R=k[X,Y,Z]/(X^{n+1}-YZ)$ with $n\geq 0$. The pairwise non-isomorphic modules $M_m:=\Syz_R(X^m,Z)$ for $m=1,\ldots,n$ give a complete list of 
representatives of the isomorphism classes of indecomposable, non-free, maximal Cohen-Macaulay modules. 
Moreover, we have $M_m^{\vee}\cong M_{n+1-m}$ and $M_m\cong (X^m,Y)$.
\end{thm}

\begin{thm}\label{syzdn}
Let $R:=k[X,Y,Z]/(X^2+Y^{n-1}+YZ^2)$ and $n\geq 4$. The pairwise non-isomorphic modules 
$$\begin{aligned}
M_1 &= \Syz_R(X,Y), &&\\
M_m &= \Syz_R(X,Y^{m/2},Z) && \text{if }m\in\{2,\ldots,n-2\}\text{ is even},\\
M_m &= \Syz_R(X,Y^{(m+1)/2},YZ) && \text{if }m\in\{2,\ldots,n-2\}\text{ is odd},\\
M_{n-1} &= \Syz_R(X,Z-iY^{(n-2)/2}) && \text{and}\\
M_n &= \Syz_R(X,Z+iY^{(n-2)/2}) && \text{if }n\text{ is even, or}\\
M_{n-1} &= \Syz_R(Z,X+iY^{(n-1)/2}) && \text{and}\\
M_n &= \Syz_R(Z,X-iY^{(n-1)/2}) && \text{if }n\text{ is odd}
\end{aligned}$$
give a complete list of representatives of the isomorphism classes of indecomposable, non-free, maximal Cohen-Macaulay modules. Moreover, these modules 
are self-dual with the exception $M_{n-1}^{\vee}\cong M_n$ if $n$ is odd. The determinants of $M_{n-1}^{\sim}|_U$ and $M_n^{\sim}|_U$ are $M_1^{\sim}|_U$. 
For the rank one modules we have the isomorphisms $M_1\cong (X,Y)$, $M_{n-1}\cong (X,Z-iY^{(n-2)/2})$ and $M_n\cong (X,Z+iY^{(n-2)/2})$ if $n$ is 
even and $M_{n-1}\cong (Z,X-iY^{(n-1)/2})$ and $M_n\cong (Z,X+iY^{(n-1)/2})$ if $n$ is odd.
\end{thm}

\begin{thm}\label{syze6}
Let $R:=k[X,Y,Z]/(X^2+Y^3+Z^4).$ The pairwise non-isomorphic modules 
\begin{align*}
M_1 &= \Syz_R(X,Y,Z),\\
M_2 &= \Syz_R(X,Y^2,YZ,Z^2),\\
M_3 &= \Syz_R(iX+Z^2,Y^2,YZ),\\
M_4 &= \Syz_R(-iX+Z^2,Y^2,YZ),\\
M_5 &= \Syz_R(-iX+Z^2,Y),\\
M_6 &= \Syz_R(iX+Z^2,Y)
\end{align*}
give a complete list of representatives of the isomorphism classes of indecomposable, non-free, maximal Cohen-Macaulay modules. Moreover, 
the modules $M_1$ and $M_2$ are self-dual, $M_3^{\vee}\cong M_4$ and $M_5^{\vee}\cong M_6$. The Determinants of $M_3^{\sim}|_U$ resp. $M_4^{\sim}|_U$ are 
$M_5^{\sim}|_U$ resp. $M_6^{\sim}|_U$. For the rank one modules we have the isomorphisms $M_5\cong (iX+Z^2,Y)$ and $M_6\cong (-iX+Z^2,Y)$.
\end{thm}

\begin{thm}\label{syze7}
Let $R:=k[X,Y,Z]/(X^2+Y^3+YZ^3).$ The pairwise non-isomorphic modules 
\begin{align*}
M_1 &= \Syz_R(X,Y,Z),\\
M_2 &= \Syz_R(X,Y^2,YZ,Z^2),\\
M_3 &= \Syz_R(XY,XZ,Y^2,YZ^2,Z^3),\\
M_4 &= \Syz_R(X,Y^2,YZ),\\
M_5 &= \Syz_R(XY,XZ,Y^2,YZ^2),\\
M_6 &= \Syz_R(X,Y,Z^2),\\
M_7 &= \Syz_R(X,Y)
\end{align*}
give a complete list of representatives of the isomorphism classes of indecomposable, non-free, maximal Cohen-Macaulay modules. Moreover, all $M_j$ 
are self-dual and $\Det(M_j^{\sim}|_U)\cong M_7^{\sim}|_U$ for $j\in\{4,5,7\}$. The rank one module $M_7$ is isomorphic to the ideal $(X,Y)$.
\end{thm}

\begin{thm}\label{syze8}
Let $R:=k[X,Y,Z]/(X^2+Y^3+Z^5).$ The pairwise non-isomorphic modules 
\begin{align*}
M_1 &= \Syz_R(X,Y,Z),\\
M_2 &= \Syz_R(X,Y^2,YZ,Z^2),\\
M_3 &= \Syz_R(XY,XZ,Y^2,YZ^2,Z^3),\\
M_4 &= \Syz_R(XY,XZ^2,Y^3,Y^2Z,YZ^3,Z^4),\\
M_5 &= \Syz_R(XY^2,XYZ^2,XZ^4,Y^4,Y^3Z,Y^2Z^3,Z^5),\\
M_6 &= \Syz_R(X,Y^2,YZ,Z^3),\\
M_7 &= \Syz_R(XY,XZ,Y^2,YZ^2,Z^4),\\
M_8 &= \Syz_R(X,Y,Z^2)
\end{align*}
give a complete list of representatives of the isomorphism classes of indecomposable, non-free, maximal Cohen-Macaulay modules. 
Moreover, all $M_j$ are self-dual.
\end{thm}

Note that most of the maximal Cohen-Macaulay can be realized as first syzygy modules of monomial ideals. Moreover, for the modules of rank at least two 
these ideals tend to be $R_+$-primary.

\section{Relations between the maximal Cohen-Macaulay modules over the ADE-rings}
Considering the rings of type $ADE$ as rings of invariants under finite subgroups of $\SL_2(k)$, 
we get ring inclusions among them by the actions of normal subgroups as shown in Table \ref{diaade} below.
\begin{center}
\begin{table}[ht]
\begin{tabular}{c}
\xymatrix{
 & \Z/(4n)\ar@^{(->}[r] & \D_{2n} & \\
\Z/(t)\ar@^{(->}[ru]^{t|4n}\ar@^{(->}[rd]^{t|2n} & & & \\
 & \Z/(2n)\ar@^{(->}[uu]\ar@^{(->}[r]& \D_n\ar@^{(->}[uu] & & \\
\Z/(2)\ar@^{(->}[uu]^{2|t}\ar@^{(->}[ur]\ar@^{(->}[dr]\ar@^{(->}[ddr] & & & \\
 & \Z/(4)\ar@^{(->}[uu]^{2|n}\ar@^{(->}[r]& \D_2\ar@^{(->}[r] & \T\ar@^{(->}[d] \\
 & \I & & \Obb
}
\end{tabular}
\caption{Inclusions of the finite subgroups (up to conjugation) of $\SL_2(\C)$. All direct inclusions are inclusions of normal subgroups.\label{diaade}}
\end{table}
\end{center}
We use the generators of the finite subgroups of $\SL_2(\C)$, given in \cite{mcmsur}. Note that these representations are well-defined in positive 
characteristic, if the group order is invertible (modulo the characteristic).

Let us very briefly recall some facts from invariant theory. If $G$ is a group, acting on a ring $R$ by ring automorphisms and $H$ is a normal subgroup 
of $G$, the ring of invariants $R^G$ is a subring of $R^H$, the quotient $G/H$ acts on $R^H$ and one has $R^G=\left(R^H\right)^{G/H}$. 
Let $X:=\Spec\left(R^G\right)$ and $Y=\Spec\left(R^H\right)$ (or the punctured resp. projective spectra). The inclusion $R^G\inj R^H$ induces a morphism $\iota:Y\ra X$. 
Given an $\Oc_X$-module $\Fc$, its pull-back $\iota\pb(\Fc)$ is an $\Oc_Y$-module. If $R=k[x,y]$ and $G\subsetneq\SL_2(k)$ is finite, we want to understand 
what happens to the sheaves associated to the indecomposable, non-free, maximal Cohen-Macaulay modules under these pull-backs. For this purpose 
let $i:R\inj S$ with $R:=k[x,y]^G=k[U,V,W]/F_1(U,V,W)$ and  $S:=k[x,y]^H=k[X,Y,Z]/F_2(X,Y,Z)$
be an inclusion between two (different) rings of type $ADE$, where $H$ is a normal subgroup of $G$. This inclusion is given by sending $U$, $V$, $W$ to 
homogeneous polynomials in the variables $X$, $Y$, $Z$ and induces a morphism 
$$\iota:\Sc:=\punctured{S}{(X,Y,Z)}\lra\Rc:=\punctured{R}{(U,V,W)}.$$
Let $M:=\Syz_R(f_1,\ldots,f_{m+1})$ be one of the indecomposable, non-free, maximal Cohen-Macaulay $R$-modules from Theorems \ref{syzan}-\ref{syze8}. 
We obtain the isomorphism
$$\iota\pb(M^{\sim}|_{\Rc})\cong\Syz_S(i(f_1),\ldots,i(f_{m+1}))^{\sim}|_{\Sc}.$$
Now, the $\Oc_{\Sc}$-module $\Syz_S(i(f_1),\ldots,i(f_{m+1}))^{\sim}|_{\Sc}$ has to be the sheaf associated to a (not necessarily indecomposable) maximal 
Cohen-Macaulay $S$-module $N$ and we want to have an explicit description of $N$. To get this description, we do the following manipulations on 
$P:=\Syz_S(i(f_1),\ldots,i(f_{m+1}))$ which change $P$ only by an isomorphism (cf. Proposition \ref{manipulationlemma}).

\begin{itemize}
 \item[Step 1:] If the $i(f_j)$ have a common factor $A$, replace all $i(f_j)$ by $i(f_j)/A$.
 \item[Step 2:] If $B$ is a common factor of $i(f_1),\ldots,\check{i(f_j)},\ldots,i(f_{m+1})$ and the ideal generated by $B$ and the $i(f_j)$ 
 is $(X,Y,Z)$-primary, replace all $i(f_k)$, $k\neq j$, by $i(f_k)/B$.
\end{itemize}

In both steps, we allow to replace $i(f_j)$ by $i(f_j)+g\cdot i(f_l)$ for $g\in S$ and $j\neq l$ or by a constant multiple $\lambda\cdot i(f_j)$ with 
$\lambda\in k^{\times}$.

\begin{exa}
Since the group $\Z/(2n)$ is a normal divisor of the group $\D_n$, we obtain a ring inclusion
$$\phi:k[X,Y,Z]/(X^2+Y^{n+1}+YZ^2)\lra k[U,V,W]/(U^{2n}-VW),$$
given by $X\mapsto \tfrac{i}{2}\cdot U(V-W)$, $Y\mapsto \sqrt[n+1]{-1}\cdot U^2$, $Z\mapsto \tfrac{1}{2}\cdot(V+W)$.
Lets name the rings above by their type of singularity, hence $\phi:D_{n+2}\ra A_{2n-1}$. Then $\phi$ induces a map $\psi$ between the punctured 
spectra of these two rings, called $\Dc_{n+2}$ and $\Ac_{2n-1}$. Now consider the module $M:=\Syz_{D_{n+2}}(X,Y,Z)$, which is the module $M_2$ from 
Theorem \ref{syzdn}. We denote the restriction to $\Dc_{n+2}$ of the sheafification of $\Syz_{D_{n+2}}(X,Y,Z)$ by $\Syz_{\Dc_{n+2}}(X,Y,Z)$ (and similarly 
for sheaves over $\Ac_{2n-1}$ coming from syzygy modules over $A_{2n-1}$. Then (provided $p\neq 2$)
\begin{align}
\nonumber \psi\pb(\Syz_{\Dc_{n+2}}(X,Y,Z)) \cong & \Syz_{\Ac_{2n-1}}(\tfrac{i}{2}\cdot U(V-W),\sqrt[n+1]{-1}\cdot U^2,\tfrac{1}{2}\cdot(V+W))\\
\nonumber \cong & \Syz_{\Ac_{2n-1}}(U(V-W),U^2,V+W)\\
\nonumber \cong & \Syz_{\Ac_{2n-1}}(UV,U^2,V+W)\\
\label{step2} \cong & \Syz_{\Ac_{2n-1}}(V,U,V+W)\\
\nonumber \cong & \Syz_{\Ac_{2n-1}}(V,U,W)\\
\nonumber \cong & \Syz_{\Ac_{2n-1}}(V,U)\oplus\Syz_{\Ac_{2n-1}}(W,U),
\end{align}
where the isomorphism (\ref{step2}) is due to ``Step 2'', since $(U,V+W)$ is $(U,V,W)$-primary.
\end{exa}

Computing $\phi\pb(\Syz_R(f_1,\ldots,f_m)^{\sim}|_{\Rc})$ for all $R$ of type $ADE$, for all possible $\phi$ as shown in Table \ref{diaade} and for all 
$\Syz_R(f_1,\ldots,f_m)$ from Theorems \ref{syzan}-\ref{syze8} as in the previous example, we obtain in surprisingly many cases a first syzygy module which 
is listed explicitly in the Theorems \ref{syzan}-\ref{syze8}.

If this is not the case, we compute generators for $P=\Syz_S(i(f_1),\ldots,i(f_{m+1}))$ and look for relations among these generators. These relations will detect the module $P^{\vee}$. 
To compute generators for $P$, we will work over the factorial domain $k[x,y]$ (the ring on which the group acts) as follows. Let $(a_1,\ldots,a_{m+1})\in P$. Then treat the relation 
$\sum a_j\cdot i(f_j)=0$ over $k[x,y]$, meaning that we write the polynomials $i(f_j)$ in the variables $x$, $y$. Then one cancels the common 
multiple of the $i(f_j)$, written in $x$, $y$. The result is a relation $\sum a_j\cdot g_j(x,y)=0$, where the $g_j$ are coprime. From this relation 
one can compute the $a_j$ (as elements in $S$!), using the fact that $k[x,y]$ is factorial. We give an example.

\begin{exa}
From the inclusion $\Z/(3)\ra \Z/(9)$, we obtain a ring inclusion
$$\phi:A_8=k[X,Y,Z]/(X^9-YZ)\ra A_2=k[U,V,W]/(U^3-VW),$$
given by $X\mapsto U$, $Y\mapsto V^3$, $Z\mapsto W^3$. Using the same notation philosophy as before, the map $\phi$ induces a morphism 
$\psi:\Ac_2\ra\Ac_8$ between the punctured spectra. The module $\Syz_{A_8}(X,Z)$ appears in Theorem \ref{syzan} and 
$$\psi\pb(\Syz_{\Ac_8}(X,Z)) \cong \Syz_{\Ac_2}(U,W^3).$$
In this case we cannot use our reduction rules, since the ideal $(U,W)$ is not $(U,V,W)$-primary. Therefore, we use the process described 
before this example. The group $\Z/(3)$ acts on $k[x,y]$ by $x\mapsto \epsilon x$ and $y\mapsto \epsilon^2 y$, where $\epsilon$ is a primitive 3rd root 
of unity. The ring of invariants is $A_2$ with $U=xy$, $V=x^3$, $W=y^3$. Now pick an element $(s_1,s_2)\in\Syz_{A_2}(U,W^3)$, hence
\begin{equation}\label{relation}
s_1\cdot U+s_2\cdot W^3 = 0.
\end{equation}
On $k[x,y]$ this relation becomes
$$s_1\cdot xy+s_2\cdot y^9 = 0$$
and is equivalent to 
\begin{equation}\label{relation2}
s_1\cdot x+s_2\cdot y^8 = 0.
\end{equation}
Over $k[x,y]$ the solutions of (\ref{relation2}) have the form 
\begin{equation}\label{solutions}
(s_1,s_2)=c\cdot (-y^8,x)
\end{equation}
for some $c\in k[x,y]$. The solutions of (\ref{relation}) are of the form (\ref{solutions}) with $c$ such that $cy^8$ as well as $cx$ are polynomials in 
$x^3$, $y^3$ and $xy$. In this case all monomials appearing in $c$ have to be multiples of either $y$ or $x^2$. Hence, the set of solutions of 
(\ref{relation}) is generated as $A_2$-module by $\tau_1:=y(-y^8,x)=(-W^3,U)$ and $\tau_2:=x^2(-y^8,x)=(-U^2W^2,V)$. These generators satisfy the global 
relation 
$$V\cdot \tau_1-U\cdot\tau_2=0.$$
This shows 
$$\Syz_{A_2}(U,W^3)\cong\Syz_{A_2}(U,V)^{\vee}\cong\Syz_{A_2}(U,W),$$
where the second isomorphism is due to Theorem \ref{syzan}.
\end{exa}

Using this approach and Macaulay2\nocite{M2}, one can compute for each ring $R$ of type $ADE$ and for each (representative of one of the 
isomorphism classes of the) indecomposable, maximal Cohen-Macaulay $R$-module its pull-backs along all morphisms from Table \ref{diaade}. 
These computations are explicitly done in \cite[Chapter 5]{drdaniel}. At this point we only state the results that are necessary for the later 
argumentation.
\begin{center}
\begin{table}[ht]
$$\begin{array}{c||c|c|c}
\Dc_{n+2} & \Ac_{2n-1} & \Ac_1 & \Ac_{n-1}\\ \hline
\Mcc_1 & \Oc_{\Ac_{2n-1}} & \Oc_{\Ac_1} & \Oc_{\Ac_{n-1}}\\
\Mcc_{2m} & \Mcc_{2m-1}\oplus \Mcc_{2(n-m)+1} & \Mcc_1^2 & \Mcc_r\oplus\Mcc_{n-r},\text{ }r:=2m-1\\
\Mcc_{2m-1} & \Mcc_{2m-2}\oplus \Mcc_{2(n-m)+2} & \Oc_{\Ac_1}^2 & \Mcc_r\oplus\Mcc_{n-r},\text{ }r:=2m-2\\
\Mcc_{n+1},\text{ $n$ even} & \Mcc_n & \Oc_{\Ac_1}^2 & \Oc_{\Ac_{n-1}}\\
\Mcc_{n+2},\text{ $n$ even} & \Mcc_n & \Oc_{\Ac_1}^2 & \Oc_{\Ac_{n-1}}\\
\Mcc_{n+1},\text{ $n$ odd} & \Mcc_n & \Mcc_1 & \Oc_{\Ac_{n-1}}\\
\Mcc_{n+2},\text{ $n$ odd} & \Mcc_n & \Mcc_1  & \Oc_{\Ac_{n-1}}
\end{array}$$
\caption{The table shows the pull-back behavior of the $\Mcc_i$ on $\Dc_{n+2}$.}
\label{tabpbd}
\end{table}
\begin{table}[ht]
$$\begin{array}{c|c|c|c|c|c|c}
\Ec_6 & \Mcc_1 & \Mcc_2 & \Mcc_3 & \Mcc_4 & \Mcc_5 & \Mcc_6 \\ \hline\hline
\Ac_1 & \Mcc_1^2 & \Oc_{\Ac_1}^3 & \Mcc_1^2 & \Mcc_1^2 & \Oc_{\Ac_1} & \Oc_{\Ac_1}
\end{array}$$
\caption{The table shows the pull-back behavior of the $\Mcc_i$ on $\Ec_{6}$.}
\label{tabpbe6}
\end{table}
\begin{table}[ht]
$$\begin{array}{c|c|c|c|c|c|c|c}
\Ec_7 & \Mcc_1 & \Mcc_2 & \Mcc_3 & \Mcc_4 & \Mcc_5 & \Mcc_6 & \Mcc_7\\ \hline\hline
\Ac_1 & \Mcc_1^2 & \Oc_{\Ac_1}^3 & \Mcc_1^4 & \Oc_{\Ac_1}^2 & \Oc_{\Ac_1}^3 & \Mcc_1^2 & \Oc_{\Ac_1}
\end{array}$$
\caption{The table shows the pull-back behavior of the $\Mcc_i$ on $\Ec_{7}$.}
\label{tabpbe7}
\end{table}
\begin{table}[ht]
$$\begin{array}{c|c|c|c|c|c|c|c|c}
\Ec_8 & \Mcc_1 & \Mcc_2 & \Mcc_3 & \Mcc_4 & \Mcc_5 & \Mcc_6 & \Mcc_7 & \Mcc_8 \\ \hline\hline
\Ac_1 & \Mcc_1^2 & \Oc_{\Ac_1}^3 & \Mcc_1^4 & \Oc_{\Ac_1}^5 & \Mcc_1^6 & \Oc_{\Ac_1}^3 & \Oc_{\Ac_1}^4 & \Mcc_1^2\\
\end{array}$$
\caption{The table shows the pull-back behavior of the $\Mcc_i$ on $\Ec_{8}$.}
\label{tabpbe8}
\end{table}
\end{center}
Using these direct computations, we can show that the Frobenius pull-backs of the syzygy modules of the maximal ideals stay indecomposable in almost all 
characteristics if the singularity is of type $D$ or $E$. First recall the following lemma whose proof can be found in \cite[Lemma 1.2.6]{frobsplit}.

\begin{lem}
If $X$ is separable and $k$ is algebraically closed, we have $$\fpb{}(\Lc)\cong\Lc^p$$ for every invertible sheaf $\Lc$ on $X$.
\end{lem}

\begin{lem}\label{fpbindec}
Let $U$ be the punctured spectrum of a two-dimensional ring $R$ of type $D$ or $E$ and let $\TF:U\ra U$ be the (restricted) Frobenius morphism. Assume that the 
order of the group corresponding to the singularity is invertible in the ground field. Then $\fpb{e}(\Syz_U(\mm))$ is indecomposable for all $e\in\N$.
\end{lem}

\begin{proof}
Recall from Section \ref{secmatfac} that $\Syz_U(\mm)$ is isomorphic to $\Mcc_2$ if $R$ is of type $D$ and $\Mcc_1$ if $R$ is of type $E$. Assume first that $R$ is of type 
$D_{2n+2}$ or $E$. If $\fpb{e}(\Syz_U(\mm))$ splits for some $e\geq 1$, its pull-back to $\Ac_1$ (along the morphism $\psi$, induced by the inclusion into $A_1$) gets trivial, since all line bundles pull back to the structure 
sheaf by Tables \ref{tabpbd}-\ref{tabpbe8}. But the pull-back of $\Syz_U(\mm)$ along $\psi$ is non-trivial (in fact, it is always $\Mcc_1^2$ by the direct computations) and its Frobenius pull-back stays non-trivial in odd 
characteristics. This gives a contradiction, since the two types of pull-backs commute.

Assume that $R$ is of type $D_{n+2}$ with $n$ odd. Since $\fpb{e}(\Syz_U(\mm))$ might split as $\Mcc_{n+1}\oplus\Mcc_{n+2}$ and 
both pull back to $\Mcc_1$ on $\Ac_1$, the argument above gives no contradiction. In this case, we have to compare the pull-backs to $\Ac_{n-1}$. Then 
all line bundles from $\Dc_{n+2}$ pull back to the structure sheaf of $\Ac_{n-1}$, but $\psi\pb(\Mcc_2)\cong\Mcc_1\oplus\Mcc_{n-1}$. The Frobenius pull-backs 
on the right hand side do not get trivial, since $n$ is invertible modulo $p$.

Moreover, if $R$ is of type $D$, the Frobenius pull-backs of $\Mcc_2$ are again of the form $\Mcc_{2m}$, since the $\Mcc_{2m-1}$ get trivial after pulling them 
back to $\Ac_1$.
\end{proof}

By the previous lemma all Frobenius pull-backs of $\Syz_U(\mm)$ are again indecomposable of rank two. Since in all cases 
$\Det(\fpb{e}(\Syz_U(\mm)))\cong\Oc_U$, we obtain directly $\fpb{e}(\Syz_U(\mm))\cong\Syz_U(\mm)$ if $R$ is of type $E_6$, since in this case there is 
only one isomorphism class of indecomposable rank two modules, whose representation is a syzygy module of an $\mm$-primary ideal. 
In the $E_7$ case $\fpb{e}(\Syz_U(\mm))$ has to be isomorphic to either $\Mcc_1$ or $\Mcc_6$ and in the $E_8$ case $\fpb{e}(\Syz_U(\mm))$ has to be isomorphic 
to either $\Mcc_1$ or $\Mcc_8$. In both cases this depends on the characteristic as we will see in the next section.

\section{Detecting isomorphism classes by the Hilbert-series}
Now we need an invariant of the isomorphism classes of the indecomposable, maximal Cohen-Macaulay $R$-modules. It turns out that the Hilbert-series does 
quite a good job in this case.

For example if $R$ is of type $E_8$, there are two isomorphism classes of indecomposable, maximal Cohen-Macaulay $R$-modules of rank two, namely 
$$M_1 = \Syz_R(X,Y,Z) \qquad \text{and} \qquad M_8 = \Syz_R(X,Y,Z^2).$$

Their Hilbert-series are given by
$$\lK_{M_1}(t) = \frac{t^{16}+t^{21}+t^{25}+t^{30}}{(1-t^{10})(1-t^6)} \qquad\text{and}\qquad \lK_{M_8}(t) = \frac{t^{22}+t^{25}+t^{27}+t^{30}}{(1-t^{10})(1-t^6)}.$$

Since by Lemma \ref{fpbindec} for all primes $p\geq 7$, the modules $M^e:=\Syz_R\left(X^{p^e},Y^{p^e},Z^{p^e}\right)$ are indecomposable, maximal 
Cohen-Macaulay modules, their Hilbert-series have to be, up to a factor $t^l$, the Hilbert-series of either $M_1$ or $M_8$. To compute the Hilbert-series 
of $M^e$ in dependence on the two parameters $p$ and $e$, we will make use of the following statements.

\begin{lem}\label{ses}
Let $k$ be a field and $$R:=k[X,Y,Z]/(X^d-F(Y,Z))$$ with $\Deg(X)=\alpha$, $\Deg(Y)=\beta$, $\Deg(Z)=\gamma$ and $F\in k[Y,Z]$ homogeneous of degree 
$d\alpha$ ($d$, $\alpha$, $\beta$, $\gamma\in\N_{\geq 1}$). Let $a\in\N_{\geq 1}$ and write $a=d\cdot q +r$ with $0\leq r\leq d-1$ and $q\in\N$.
For all $s\in\Z$ we have a short exact sequence
\begin{align*}
0 & \lra \Syz_R(X^{a+d-2r},Y^b,Z^c)(s-\alpha\cdot r)\\
 & \stackrel{\psi}{\lra} \Syz_R(X^{a-r},Y^b,Z^c)(s-\alpha\cdot r)\dirsum\Syz_R(X^{a+d-r},Y^b,Z^c)(s)\\
 & \stackrel{\phi}{\lra} \Syz_R(X^a,Y^b,Z^c)(s)\lra 0,
\end{align*}
where the maps are defined via 
\begin{align*}
\psi(h_1,h_2,h_3):= & ((X^{d-r}\cdot h_1,h_2,h_3),(-h_1,-X^r\cdot h_2,-X^r\cdot h_3))\\
\phi((f_1,f_2,f_3),(g_1,g_2,g_3)):= &(f_1+X^{d-r}\cdot g_1,X^r\cdot f_2+g_2,X^r\cdot f_3+g_3). 
\end{align*}
\end{lem}

\begin{proof}
The injectivity of $\psi$ is clear and the exactness at the middle spot is straightforward. The proof that $\phi$ is surjective is essentially the 
same as in \cite[Lemma 1.1]{miyaoka}. See also \cite[Chapter 4]{drdaniel} for a detailed proof and a slight generalization.
\end{proof}

\begin{rem}
Note that the two first syzygy modules $\Syz_R(X^{a-r},Y^b,Z^c)$ and $\Syz_R(X^{a+d-r},Y^b,Z^c)$ appearing in the middle spot are already defined over $k[Y,Z]$, 
hence they split by Hilberts Syzygy Theorem as a direct sum of two (degree shifted) copies of $R$.
\end{rem}

\begin{thm}\label{HilbSer} The notations are the same as in Lemma \ref{ses}. For $l\in\N$ we use the abbreviation $\Sc(l):=\Syz_R(X^l,Y^b,Z^c)$. 
Then the Hilbert-series of $\Sc(a)=\Syz_R(X^a,Y^b,Z^c)$ is given by
\begin{equation*}
\lK_{\Sc(a)}(t) = \frac{(t^{\alpha\cdot r}-t^{\alpha\cdot d})\cdot \lK_{\Sc(a-r)}(t)+(1-t^{\alpha\cdot r})\lK_{\Sc(a+d-r)}(t)}{1-t^{\alpha\cdot d}}.
\end{equation*}
\end{thm}

\begin{proof}
 Let $a':=a+d-2r=dq+d-r$ and $r':=d-r$. Since $a'+d-2r'=a$, Theorem \ref{ses} yields
\begin{align*}
 \lK_{\Sc(a)}(t) & = t^{\alpha\cdot r}\lK_{\Sc(a-r)}(t)+\lK_{\Sc(a+d-r)}(t)-t^{\alpha\cdot r}\lK_{\Sc(a')}(t)\quad\text{and}\\
 \lK_{\Sc(a')}(t) & = t^{\alpha\cdot r'}\lK_{\Sc(a-r)}(t)+\lK_{\Sc(a+d-r)}(t)-t^{\alpha\cdot r'}\lK_{\Sc(a)}(t).
\end{align*}
Substituting $\lK_{\Sc(a')}(t)$ in the first formula and solving for $\lK_{\Sc(a)}(t)$ gives the result.
\end{proof}

Turning back to the question of computing the Hilbert-series of 
$$M^e=\Syz_R\left(X^{p^e},Y^{p^e},Z^{p^e}\right)$$
if $R$ is type $E_8$, we obtain via Theorem \ref{HilbSer}
$$\lK_{M^e}(t) = \frac{(t^{15}-t^{30})\cdot \lK_{\Syz_R\left(X^{p^e-1},Y^{p^e},Z^{p^e}\right)}(t)+(1-t^{15})\lK_{\Syz_R\left(X^{p^e+1},Y^{p^e},Z^{p^e}\right)}(t)}{1-t^{30}}.$$

This shows that it is enough to compute the Hilbert-series of the $k[Y,Z]$-modules
\begin{align*}
\Syz_{k[Y,Z]}\left(X^{p^e-1},Y^{p^e},Z^{p^e}\right)\cong & \Syz_{k[Y,Z]}\left((-Y^3-Z^5)^{\frac{p^e-1}{2}},Y^{p^e},Z^{p^e}\right),\\
\Syz_{k[Y,Z]}\left(X^{p^e+1},Y^{p^e},Z^{p^e}\right)\cong & \Syz_{k[Y,Z]}\left((-Y^3-Z^5)^{\frac{p^e+1}{2}},Y^{p^e},Z^{p^e}\right).
\end{align*}

By Hilbert's Syzygy Theorem, modules of the form 
\begin{equation}\label{syzsplit}
\Syz_{k[Y,Z]}\left((Y^3+Z^5)^{\frac{p^e\mp 1}{2}},Y^{p^e},Z^{p^e}\right)
\end{equation}
split as $k[Y,Z](m)\oplus k[Y,Z](n)$ for suitable $m$, $n\in\Z$. 

The question now is how to compute the degree shifts $m$ and $n$. If the quotients were of the form $k[Y,Z]/((Y+Z)^{d_1},Y^{d_2},Z^{d_3})$, we could make use of 
Han's $\delta$-function, which computes the difference of the degree shifts in this situation. It has a unique continuous extension to $[0,\infty)^3$ 
satisfying $\delta(pt)=p\cdot\delta(t)$ for $\chara(k)=p>0$, $t\in[0,\infty)^3$ and can be computed with Han's Theorem \ref{hansthm}. 
This is useful for us since one can reduce the non-standard-graded situation to a standard-graded case as we will see in a moment.

\begin{defi}
Let $$L_{\text{odd}}:=\left\{u\in\N^3|u_1+u_2+u_3\text{ is odd}\right\}$$ be the odd lattice.
\end{defi}

\begin{thm}[Han]\label{hansthm}
Let $t=(t_1,t_2,t_3)\in[0,\infty)^3.$ If the $t_i$ do not satisfy the strict triangle inequality (w.l.o.g. $t_1\geq t_2+t_3$), we have $\delta(t)=t_1-t_2-t_3.$
If the $t_i$ satisfy the strict triangle inequality and there are $s\in \Z$ and $u\in L_{odd}$ with $\left\Vert p^st-u\right\Vert_1<1$, then there is such a pair $(s,u)$ with minimal $s$ and with this pair $(s,u)$ we get
$$\delta(t)=\frac{1}{p^s}\cdot\left(1-\left\Vert p^st-u\right\Vert_1\right).$$ Otherwise, we have $\delta(t)=0$.
\end{thm}

The proof and more properties can be found in \cite{handiss} as well as in \cite{mason}.

To make use of this result in a non-standard-graded setting, lets fix some notations.

\begin{setup}
Let $k$ be a field of characteristic $p>0$, let $R:=k[U,V]$ equipped with the positive grading $\Deg(U)=\alpha$, $\Deg(V)=\beta$. Fix $a$, $b\in\N_{\geq 1}$ 
such that $a\cdot \alpha=b\cdot \beta$. For $a_1$, $a_2$, $a_3\in\N_{\geq 1}$ the $R$-module 
$$\Syz_R((U^a+V^b)^{a_1},U^{a_2},V^{a_3})$$
splits by Hilbert's Syzygy Theorem as $R(m)\oplus R(n)$ for some $m$, $n\in\Z$. Let
$$\tau(a_1,a_2,a_3):=|m-n|.$$
This function $\tau$ has a unique continuous extension to $[0,\infty)^3$, again with the property $\tau(pt)=p\cdot\tau(t)$. 
See \cite[Chapter 2]{drdaniel} for the details.
\end{setup}

\begin{lem}\label{deltaquasi} With the notations from above and $t\in[0,\infty)^3$ we have 
$$\tau(t_1,t_2,t_3)=a\cdot\alpha\cdot\delta\left(t_1,\frac{t_2}{a},\frac{t_3}{b}\right).$$
\end{lem}

\begin{proof}
For $t\in\N^3$ both sides multiply by $p$ when $t$ is replaced by $pt$. Since rationals of the form $\tfrac{na}{p^j}$ resp. $\tfrac{nb}{p^j}$ are dense in $[0,\infty)$, it suffices 
to prove the statement for $t=(t_1,at_2,bt_3)$ with $t_i\in\N$. Let $S$ be the standard-graded polynomial ring $k[X,Y]$. Consider the map 
$\phi:S\rightarrow R$, $X\mapsto U^a$, $Y\mapsto V^b$, which is homogeneous of degree $a\cdot \alpha$. Denote by $\psi$ the induced map on the projective spectra. 
We obtain 
$$\xymatrix{
\psi\pb(\Syz_{\Proj(S)}((X+Y)^{t_1},X^{t_2},Y^{t_3}))\ar[r]^{\cong}\ar[d]^{\cong} & \Syz_{\Proj(R)}((U^a+V^b)^{t_1},U^{at_2},V^{bt_3})\\
\psi\pb(\Oc_{\Proj(S)}(l)\oplus\Oc_{\Proj(S)}(l+\delta))\ar[r]^{\cong} & \Oc_{\Proj(R)}(a\alpha l)\oplus\Oc_{\Proj(S)}(a\alpha (l+\delta))
}$$
for some $l\in\Z$ and $\delta=\delta(t_1,t_2,t_3)$. Now the claim follows, since we have the isomorphism
$$\Syz_{\Proj(R)}((U^a+V^b)^{t_1},U^{at_2},V^{bt_3})\cong\Oc_{\Proj(R)}(m)\oplus\Oc_{\Proj(R)}(m+\tau)$$
for some $m\in\Z$ and $\tau=\tau(t_1,at_2,bt_3)$.
\end{proof}

This enables to prove the following.

\begin{lem}\label{syzpbe8}
\begin{align*}
 & \Syz_R(X^{p^e},Y^{p^e},Z^{p^e})\\
\cong & \left\{\begin{aligned} & \Syz_R(X,Y,Z)\left(-\frac{31}{2}(p^e-1)\right) && \text{if } p^e\text{ mod }30\in \{\pm 1,\pm 11\},\\ & \Syz_R(X,Y,Z^2)\left(-\frac{31}{2}(p^e-1)+3\right) && \text{if } p^e\text{ mod }30\in \{\pm 7,\pm 13\}.\end{aligned}\right.
\end{align*}
\end{lem}

\begin{proof}
Going back to Equation (\ref{syzsplit}), we obtain by Lemma \ref{deltaquasi} the equality 
$$|m-n|=30\cdot\delta\left(\frac{p^e\mp 1}{2},\frac{p^e}{3},\frac{p^e}{5}\right).$$
First we assume $e=1$ and $p=60\cdot l+1$ for some $l\in\N$. We want to compute 
$$\delta\left(\frac{p-1}{2},\frac{p}{3},\frac{p}{5}\right)=\delta\left(30\cdot l,10\cdot l+\frac{1}{3},6\cdot l+\frac{1}{5}\right).$$
It is easy to see that the $s$ from Han's Theorem has to be non-negative. Since the taxi-cap distance of 
$\left(30\cdot l,10\cdot l+\tfrac{1}{3},6\cdot l+\tfrac{1}{5}\right)$ to the odd lattice is given by 
$$\left\Vert\begin{pmatrix}30\cdot l \\ 10\cdot l+\frac{1}{3} \\ 6\cdot l+\frac{1}{5}\end{pmatrix}-\begin{pmatrix} 30\cdot l \\ 10\cdot l+1\\ 6\cdot l\end{pmatrix}\right\Vert_1=\frac{2}{3}+\frac{1}{5}=\frac{13}{15},$$
we obtain $\delta\left(\tfrac{p-1}{2},\tfrac{p}{3},\tfrac{p}{5}\right)=\tfrac{2}{15}$ and $|m-n|=4$. Treating the other possibilities for $p$ modulo $60$ and 
computing the values of $\delta\left(\tfrac{p-1}{2},\tfrac{p}{3},\tfrac{p}{5}\right)$ one obtains 
$$\begin{aligned}
& \lK_{\Syz_R(X^p,Y^p,Z^p)}(t)\\
= & \left\{\begin{aligned} & t^{31\cdot\frac{p-1}{2}}\cdot \lK_{\Syz_R(X,Y,Z)}(t) && \text{if } p\text{ mod }30\in \{\pm 1,\pm 11\},\\ & t^{31\cdot\frac{p-1}{2}-3}\cdot \lK_{\Syz_R(X,Y,Z^2)}(t) && \text{if } p\text{ mod }30\in \{\pm 7,\pm 13\},\end{aligned}\right.
\end{aligned}$$
This already shows 
$$\Syz_R(X^{p^e},Y^{p^e},Z^{p^e})\cong \Syz_R(X,Y,Z)(-\tfrac{31}{2}(p^e-1))$$
if $p\text{ mod }30\in \{\pm 1,\pm 11\}$. Computing the 
Hilbert-series of $\Syz_R(X^p,Y^p,Z^{2p})$ one finds
$$\lK_{\Syz_R(X^p,Y^p,Z^{2p})}(t) = \left\{\begin{aligned} & t^{37\cdot\frac{p-1}{2}}\cdot \lK_{\Syz_R(X,Y,Z^2)}(t) && \text{if } p\text{ mod }30\in \{\pm 1,\pm 11\},\\ & t^{37\cdot\frac{p-1}{2}+3}\cdot \lK_{\Syz_R(X,Y,Z)}(t) && \text{if } p\text{ mod }30\in \{\pm 7,\pm 13\}.\end{aligned}\right.$$
\end{proof}

The previous lemma shows that we have a $(0,1)$-Frobenius periodicity (on the punctured spectrum) if $p\text{ mod }30\in\{\pm 1,\pm 11\}$ and a 
$(0,2)$-Frobenius periodicity if $p\text{ mod }30\in\{\pm 7,\pm 13\}$.

Similarly, one obtains Frobenius periodicities of $\Syz(X,Y,Z)$ on the punctured spectrum for all $ADE$-rings.

\section{The Hilbert-Kunz functions of the ADE-rings}
In this section we put everything together to compute the Hlbert-Kunz functions of the $ADE$-rings.

\begin{thm}\label{hkfe8}
The Hilbert-Kunz function of $R=k[X,Y,Z]/(X^2+Y^3+Z^5)$ is given by
$$e\longmapsto\left\{\begin{aligned}
 & 2\cdot p^{2e} && \text{if } e\geq 1\text{ and }p\in\{2,3,5\},\\
 & \left(2-\frac{1}{120}\right)p^{2e}-\frac{191}{120} && \text{if }p\text{ mod }30\in\{\pm 7,\pm 13\}\text{ and }e\text{ is odd},\\
 & \left(2-\frac{1}{120}\right)p^{2e}-\frac{119}{120} && \text{otherwise.}
\end{aligned}\right.$$
\end{thm}

\begin{proof}
By Lemma \ref{syzpbe8} we know the isomorphism classes of the syzygy modules $\Syz_R(X^{p^e},Y^{p^e},Z^{p^e})$. 
This allows us to use Lemma \ref{syznontr} and we obtain
\begin{align*}
\HKF(R,p^e) &= \frac{30\cdot Q(15,10,6)\cdot p^{2e}-30\cdot Q(15,10,6)+30\cdot 120}{30\cdot 120}\\
&= \frac{239\cdot p^{2e}-239+120}{120}\\
&= \left(2-\frac{1}{120}\right)p^{2e}-\frac{119}{120},
\end{align*}
if $p^e\text{ mod }30\in\{\pm 1,\pm 11\}$ and
\begin{align*}
\HKF(R,p^e) &= \frac{30\cdot Q(15,10,6)\cdot p^{2e}-30\cdot Q(15,10,12)+30\cdot 240}{30\cdot 120}\\
&= \frac{239\cdot p^{2e}-431+240}{120}\\
&= \left(2-\frac{1}{24}\right)p^{2e}-\frac{191}{120},
\end{align*}
if $p^e\text{ mod }30\in\{\pm 7,\pm 13\}$.

If $p\in\{2,3,5\}$ we see that already $\Syz_R(X^p,Y^p,Z^p)$ splits, since in each case $X^p$, $Y^p$ or $Z^p$ is a polynomial in the other two variables. 
For example if $p=2$, we have
\begin{align*}
\Syz_R(X^2,Y^2,Z^2) &= \Syz_R(Y^3+Z^5,Y^2,Z^2)\\
&\cong R(-30)\oplus R(-32),
\end{align*}

hence $\Syz_R(X^{2^e},Y^{2^e},Z^{2^e})\cong R(-15\cdot 2^e)\oplus R(-16\cdot 2^e)$. Using Lemma \ref{syztr} one obtains 
$\HKF_R(R,p^e)=2\cdot p^{2e}$ for $e\geq 1$.
\end{proof}

Similarly, one obtains the Hilbert-Kunz functions of the rings of type $E_6$ and $E_7$. In the case of an $E_6$ singularity we already know from 
Lemma \ref{fpbindec} that
$$\Syz_R\left(X^{p^e},Y^{p^e},Z^{p^e}\right)\cong\Syz_R(X,Y,Z)\left(-\frac{13}{2}(p^e-1)\right)$$
holds for all $e\geq 0$ and all $p\geq 5$ (since there is no other possibility). 
Using Lemmata \ref{syznontr} and \ref{syztr} one obtains the following Hilbert-Kunz function.

\begin{thm}\label{hkfe6}
The Hilbert-Kunz function of $R=k[X,Y,Z]/(X^2+Y^3+Z^4)$ is given by
$$e\longmapsto\left\{\begin{aligned}
 & 2\cdot p^{2e} && \text{if } e\geq 1\text{ and }p\in\{2,3\},\\
 & \left(2-\frac{1}{24}\right)p^{2e}-\frac{23}{24} && \text{otherwise.}
\end{aligned}\right.$$
\end{thm}

If $R$ is of type $E_7$ and $p\geq 5$, we have similarly to the $E_8$ case the two possibilities 
\begin{align*}
\Syz_R(X^p,Y^p,Z^p) & \cong\Syz_R(X,Y,Z)\left(-\frac{19}{2}\cdot (p-1)\right) \text{ or}\\ 
\Syz_R(X^p,Y^p,Z^p) & \cong\Syz_R(X,Y,Z^2)\left(-\frac{19}{2}\cdot (p-1)+2\right).
\end{align*}
Computing the Hilbert-series, one sees that the first case appears if 
$p\text{ mod }24\in\{\pm 1,\pm 7\}$ yielding a $(0,1)$-Frobenius periodicity of $\Syz(X,Y,Z)$ on the punctured spectrum. If $p\text{ mod }24\in\{\pm 5,\pm 11\}$ 
the second case holds and one obtains a $(0,2)$-Frobenius periodicity on the punctured spectrum.

\begin{thm}\label{hkfe7}
 The Hilbert-Kunz function of $E_7$ is given by
$$e\longmapsto\left\{\begin{aligned}
 & 2\cdot p^{2e} && \text{if } e\geq 1\text{ and }p\in\{2,3\},\\
 & \left(2-\frac{1}{48}\right)p^{2e}-\frac{71}{48} && \text{if }p\text{ mod }24\in\{\pm 5,\pm 11\}\text{ and }e\text{ is odd},\\
 & \left(2-\frac{1}{48}\right)p^{2e}-\frac{47}{48} && \text{otherwise.}
\end{aligned}\right.$$
\end{thm}

It remains to treat the family of $D_n$ singularities. In this case we failed to compute the Hilbert-series of $\Syz_R(X^{p^e},Y^{p^e},Z^{p^e})$ 
in general. If one fixes a concrete value for $n$, it is possible to compute the Hilbert-series and everything works as in the $E$ situations. 
In the general case we change the approach to identify the isomorphism class of $\Syz_R(X^{p^e},Y^{p^e},Z^{p^e})$. 

Let $R:=k[X,Y,Z]/(X^2+Y^{n+1}+YZ^2)$ with $\Deg(X)=n+1$, $\Deg(Y)=2$, $\Deg(Z)=n$ and $\chara(k)=p\nmid 4n$. Then $R$ has an $D_{n+2}$ singularity and 
injects into a two-dimensional ring with an $A_{2n-1}$ singularity. The inclusion is given by
$$\phi:R\lra k[U,V,W]/(U^{2n}-VW)$$ 
and the pull-back behavior of the maximal Cohen-Macaulay modules can be read off Table \ref{tabpbd}. We will again use curly letters for punctured spectra 
and modules over them. By Lemma \ref{fpbindec} we have 
\begin{equation}
\fpb{e}(\Syz_{\Rc}(\mm))=\fpb{e}(\Mcc_2)\cong \Mcc_{2m}  
\end{equation}
for all $e\in\N$ and $m=m(q)\in\{1,\ldots,\tfrac{n}{2}\}$. The parameter $m=m(q)$ can be computed by considering the pull-back via $\phi$: 
$\Mcc_2$ pulls back to $\Nc_1\oplus\Nc_{2n-1}$ on $\Ac_{2n-1}$. The $e$-th Frobenius pull-back of this module is $\Nc_r\oplus\Nc_{2n-r}$ with 
$r\equiv p^e$ mod $2n$ (and $1\leq r\leq 2n-1$). Note that $r$ is odd by the assumption $p\nmid 4n$. If $r\leq n-1$, only $\Mcc_{r+1}$ pulls back to $\Nc_r\oplus\Nc_{2n-r}$ via $\phi$. 
If $r\geq n$, only $\Mcc_{2n-r-1}$ pulls back to $\Nc_r\oplus\Nc_{2n-r}$ via $\phi$. But $\Mcc_{r+1}$ and $\Mcc_{2n-r-1}$ are isomorphic on $\Rc$ as one sees 
by computing the first syzygy modules from the matrix factorizations by using different columns (cf. \cite[Section 3.5]{drdaniel}). We find 
$$\Syz_{\Rc}\left(U^{p^e},V^{p^e},W^{p^e}\right)\cong\Syz_{\Rc}(U,V^{(r+1)/2},W),$$ 
where $r\equiv p^e$ mod $2n$ with $1\leq r\leq 2n-1$. This isomorphism induces a global isomorphism
$$\Syz_R\left(U^{p^e},V^{p^e},W^{p^e}\right)\cong\Syz_R(U,V^{(r+1)/2},W)(m),$$
for some $m\in\Z$. The Hilbert-Kunz function can be computed via Lemma \ref{syznontr}.

Note that we can define all appearing syzygy modules in the case $p|4n$ as well. 
Since the computations made to obtain Table \ref{tabpbd} are correct in all odd characteristics, the above computation of the Hilbert-Kunz function 
remains correct in the case $p|4n$ with $p$ odd, if we would know that our list of maximal Cohen-Macaulay modules is complete.

By a theorem of Herzog (cf. \cite[Theorem 6.3]{leuwie}) we only need that $R$ is a direct summand of the polynomial ring $k[u,v]$, on which the 
group $\D_n$ acts, to have an one to one correspondence between indecomposable, maximal Cohen-Macaulay $R$-modules and direct summands of $k[u,v]$ as 
an $R$-module.

Leuschke and Wiegand point out in \cite[Remark 6.22]{leuwie} that $R$ is a direct summand of $k[u,v]$ and that the direct summands of $k[u,v]$ 
as an $R$-module are the same as in the non-modular case (meaning that the algebraic definition is the same, although the modules in the modular 
case might not come from representations of the group $\D_n$). Therefore, the formula above is a full description of the Hilbert-Kunz function 
for odd characteristics.

To compute the Hilbert-Kunz function for $p=2$, we can use Lemma \ref{syztr}. All in all, one obtains

\begin{thm}\label{hkfdn}
 The Hilbert-Kunz function of $D_{n+2}$ is given by (with $r\equiv p^e$ mod $2n$ for $0\leq r\leq 2n-1$)
$$e\longmapsto\left\{\begin{aligned}
 & 2\cdot p^{2e} && \text{if } e\geq 1,p=2,\\
 & \left(2-\frac{1}{4n}\right)p^{2e}-\frac{r+1}{2}+\frac{r^2}{4n} && \text{otherwise.}
          \end{aligned}\right.$$
\end{thm}

Note that our explicit knowledge of the Hilbert-Kunz functions enables us to compute the $F$-signature functions by using \cite[Theorem 11]{2thmsmcm}. 
If $R$ is of type $D$ or $E$, one has
$$\FS(R,p^e)=2p^{2e}-\HKF(R,p^e).$$

\bibliographystyle{alpha} 
\bibliography{bibfile}

\end{document}